\documentclass[reqno]{amsart}
\usepackage{epsfig,latexsym,amsfonts,amssymb,amsmath,amscd}
\usepackage{amsthm}

\usepackage{tikz-cd}
\usepackage[mathscr]{eucal}
\usepackage{mathrsfs}
\usepackage{mathbbol}
\usepackage{array}
\usetikzlibrary{matrix,arrows,decorations.pathmorphing}
\usepackage{oldgerm,units,color}

\usepackage{pst-node}
\usepackage{graphicx}

\usepackage{eepic, epic}

\usepackage{ifthen}

\def\tropa{  {\operatorname{trop}}\, }

\def\Gz{ {\Gamma}_{\zero}}

\def\ssemiring0{$s$-semiring$^\dagger$}

\usepackage{tikz}
\usepackage{epsfig,latexsym,amsfonts,amssymb,amsmath,amscd,graphics,epic}
\usepackage{amsfonts,amssymb,amsmath,amscd,amsthm}
\usepackage[mathscr]{eucal}
\usepackage{mathrsfs}

\usepackage{mathbbol}

\usepackage{oldgerm,units}
\usepackage{wrapfig,epsfig}

\usepackage{ifthen}
\usepackage{mathbbol}

\usepackage{amsthm}
\usepackage[mathscr]{eucal}
\usepackage{mathrsfs}
\usepackage{mathbbol}
\usepackage{oldgerm,units}
\usepackage{wrapfig}

\newtheorem{theorem}{Theorem}[section]

\newtheorem{example}[theorem]{Example}

\newcommand{\one}{\mathbb{1}}
\newcommand{\zero}{\mathbb{0}}

\newcommand{\Trop}{\mathbb T}

\newcommand{\tG}{\trop{G}}

\newcommand{\tT}{\trop{T}}

    \ifx\proof\undefined
    \newenvironment{proof}{
    \smallskip
    \noindent\emph{Proof.}}{\hfill\(\Box\)
    \bigskip
    } \fi

\newcommand{\ifdef}[3]{\ifthenelse{\equal{#1}{true}}{#2}{#3}}

\definecolor{lgray}{gray}{0.90}

\def\ctw{\cdot_{\operatorname{tw}}}

\def\val{\operatorname{val}}

\def\bw{\bigwedge}

\def\({\left(}
\def\){\right)}

\def\strop{  {\operatorname{strop}}\, }
\def\Trop{  {\operatorname{Trop}}}

\def\pipe{{\underset{{\ \, }}{\mid}}}

\def\vsemifield0{$\nu$-semifield$^\dagger$}
\def\vsemiring0{$\nu$-semiring$^\dagger$}

\def\pipe1{{\underset{{1}}{\mid}}}
\def\lmod1{\mathrel  \pipe1  \joinrel \joinrel =}

\def\CFunFF1{\operatorname{CFun} (F,F)}

\def\semiring0{semiring$^{\dagger}$}
\def\Semiring0{Semiring$^{\dagger}$}
\def\Semirings0{Semirings$^{\dagger}$}
\def\semidomain0{semidomain$^{\dagger}$}
\def\semifield0{semifield$^{\dagger}$}
\def\semifields0{semifields$^{\dagger}$}
\def\vsemifields0{$\nu$-semifields$^{\dagger}$}
\def\domain0{domain$^{\dagger}$}
\def\predomain0{pre-domain$^{\dagger}$}
\def\predomains0{pre-domains$^{\dagger}$}
\def\domains0{domains$^{\dagger}$}
\def\vdomains0{$\nu$-domains$^{\dagger}$}

\def\Fun{\operatorname{Fun}}

\def\domains0{domains$^\dagger$}
\def\w{\wedge}
\def\Dz{D\{z\}}

\newcommand{\etype}[1]{\renewcommand{\labelenumi}{(#1{enumi})}}
\def\eroman{\etype{\roman}}

\def\pipe{{\underset{{\tG}}{\mid}}}

\def\lmod{\mathrel  \pipe \joinrel \joinrel =}
\def\pipe{{\underset{{\tG}}{\mid}}}

\def\ealph{\etype{\alph}}

\def\tG{\mathcal G}
\def\RGnu{(R,\tG,\nu)}

\def\pSkip{\vskip 1.5mm \noindent}

\def\a{\alpha}
\newtheorem{thm}[theorem]{Theorem}

\newtheorem*{thm*}{Theorem}

\def\Mor{\operatorname{Mor}}

\newtheorem{rem}[theorem]{Remark}
\newtheorem{prop*}{Proposition}

\newtheorem{prop}[theorem]{Proposition}
\newtheorem{defn}[theorem]{Definition}
\newtheorem*{examp*}{Example}
\newtheorem*{examples*}{Examples}
\newtheorem*{remark*}{Remark}
\newtheorem*{defn*}{Definition}

\def\la{\lambda}

\def\tT{\mathcal T}
\def\Fun{\operatorname{Fun}}
\def\tTz{\tT_\zero}

\numberwithin{equation}{section}

\def\M0{M_{\zero}}

\def\supp{\operatorname{supp}}
\def\vsupp{\nu$-$\operatorname{supp}}

\def\tGz{\mathcal G_\zero}

\def\PS{P}

\def\Cong{\Phi}

\def\semirings0{semirings$^\dagger$}

\newcommand{\nPS}[1]{\PS_{(!#1)}}
\newcommand{\nPSo}[1]{\nPS{\one}}

\newcommand{\absl}[1]{|{#1}|}
\newcommand{\adj}[1]{\operatorname{adj}({#1})}

\begin{document}


\title[An informal overview of triples and systems]{An informal overview of triples and systems}

\author[L.~Rowen]{Louis Rowen}
\address{Department of Mathematics, Bar-Ilan University, Ramat-Gan 52900,
Israel} \email{rowen@math.biu.ac.il}

\subjclass[2010]{Primary   08A05,  16Y60, 12K10,  13C60, 20N20;
Secondary 03C05, 06F05, 13C10, 14T05}

\date{\today}


\keywords{bipotent, category,  congruence, dual basis,
  homology, hyperfield, linear algebra, matrix,   meta-tangible,
morphism, negation map, module,      polynomial,  prime, projective,
tensor product, semifield,  semigroup,  semiring, split,
supertropical algebra,  superalgebra, surpassing relation,
symmetrization,
  system,
triple, tropical.}

\thanks{\noindent \underline{\hskip 3cm } \\ File name: \jobname}


\begin{abstract}

We describe triples and systems, expounded  as an axiomatic
algebraic umbrella theory
 for classical algebra, tropical algebra, hyperfields, and fuzzy rings.
\end{abstract}

\maketitle





\section{Introduction}
The goal of this overview is to present an axiomatic algebraic
theory which unifies, simplifies, and ``explains'' aspects of
tropical algebra
\cite{zur05TropicalAlgebra,IR,IR1,IzhakianKnebuschRowen2011CategoriesII},
hyperfields \cite{BB,GJL,Ju,Vi}, and fuzzy rings \cite{Dr,DW,GJL} in
terms of familiar algebraic concepts. It was motivated by an attempt
to understand whether or not it is coincidental that basic algebraic
theorems are mirrored in supertropical algebra, and was spurred by
the realization that some of the same results are obtained in
parallel research on hyperfields and fuzzy rings. Our objective is
to hone in on the precise axioms that include these various
examples, formulate the axiomatic structure, describe its uses, and
review five papers \cite{Row16,AGR,JuR1,JuMR1,JuMR2} in which the
theory is developed. The bulk of this survey concerns \cite{Row16},
in which the axiomatic framework is laid out, since the other papers
build on it.

Other treatments can be found in
\cite{CC2,grandis2013homological,Lor1}. Although we deal with
general categorical issues, ours is largely a ``hands on'' approach,
emphasizing a ``negation map'' which exists in all of the
above-mentioned examples, and which often is obtained by means of a
``symmetrization'' functor. The other key ingredient is a
\textbf{surpassing relation}~$\preceq$, to replace equality in our
theorems.  (In classical mathematics, $\preceq$ is just equality.)

%
%

The quadruple $(\mathcal A, \tT, (-),\preceq)$ is called a
$\tT$-\textbf{system}. Although the investigation of systems has
centered on semirings, having grown out of tropical considerations,
it also could be used to develop a parallel Lie semi-algebra theory
(and more generally Hopf semi-algebra theory).

\subsection{Acquaintance with basic notions}$ $

One starts with a set~$\tT$  that we want to study, called the set
of \textbf{tangible
 elements}, endowed with  a partial
additive algebraic structure which however is not defined on all of
$\tT$; this is resolved by embedding $\tT$ in a larger set $\mathcal
A$ with a fuller algebraic structure.  Often $\tT$ is a
multiplicative monoid \footnote{Geometry over monoids has been a
subject of recent interest, cf.~\cite{CHWW}.}, a situation developed
by Lorscheid \cite{Lor1,Lor2} when $\mathcal A$ is
  a semiring. However, there also are examples (such as Lie algebras) lacking
associative multiplication. We usually denote a typical element of
$\tT$ as $a$, and a typical element of $\mathcal A$ as $b$.

\begin{defn}\label{modu13}
A $\tT$-\textbf{module}   over a set $\tT$ \footnote{More generally,
when $\tT$ is not a monoid, Hopf theory could play an appropriate
role.} is an additive monoid $(\mathcal A,+,\zero _\mathcal A)$
together with scalar multiplication $\tT\times \mathcal A \to
\mathcal A$ satisfying distributivity over $\tT$
\footnote{Distributivity over elements of $\tT$ is enough to run the
theory, since one can (re)define multiplication on $\mathcal A$ to
make it distributive, as seen in \cite[Theorem~2.9]{Row16}. This
rather easy result applies for instance to hyperfields such as the
phase hyperfield.}
  in the sense that
$$a(b_1+b_2) = ab_1 +ab_2$$ for $a \in \tT,$ $b _i \in \mathcal A$,
also stipulating that $a\zero_{\mathcal A} = \zero_{\mathcal A}$.

 A $\tT$-\textbf{monoid module}   over a
multiplicative monoid $\tT$ is a $\tT$-module $\mathcal A$
satisfying the extra conditions
$$\one_\tT b = b, \qquad (a_1a_2) b = a_1, \quad \forall a_i \in
\tT, \ b \in \mathcal A.$$
\end{defn}


For the sake of this exposition, we assume that $\mathcal A$ is a
$\tT$-module and
 $\tT \subseteq \mathcal A.$

A semigroup $(\mathcal A,+)$  has \textbf{characteristic~$k>0$} if
$( k+1)a  =a  $ for all $a \in \mathcal A,$ with $k \ge 1$ minimal.
$\mathcal A$ has \textbf{characteristic
 $0$} if $\mathcal A$ does not have  characteristic
 ~$k$ for any
 $k\ge 1.$

 Properties of the characteristic are
described in \cite[\S 6.4]{Row16}. Most of our major examples have
characteristic 0, but some interesting examples have characteristic
$2$ or more.

 A \textbf{\semiring0} satisfies all the axioms of ring
except the existence of a $\zero $ element and of negatives.  A
\textbf{semiring} \cite{golan92} is a \semiring0 with $\zero$.

\subsubsection{Brief overview}$ $

We introduce a formal negation map~$(-)$, which we describe in
\S\ref{trsys} after some introductory examples, such that $\tT$
generates~$\mathcal A$ additively, creating a $\tT$-\textbf{triple}
$(\mathcal A, \tT, (-))$.

When a formal negation map is not available at the outset, we can
introduce it in two ways, to be elaborated shortly:

\begin{itemize}\item Declare the negation map to be the identity, as
in the supertropical case, cf.~\S\ref{supert}.
\item  Apply symmetrization, to get the switch map, of second kind,
cf.~\S\ref{symm}.  Often \cite{AGG2} is applicable, where $\tT$
could take the role of the ``thin elements.''
\end{itemize}

The element  $b^\circ : = b (-)b$ is called a  \textbf{quasi-zero}.
We write $\mathcal A^\circ $ for $\{ b^\circ: b \in \mathcal A\},$
and usually require that $\tT \cap \mathcal A^\circ = \emptyset,$
i.e., a quasi-zero cannot be tangible.

In classical algebra, the only quasi-zero is $\zero$ itself, and
$a_1-a_1 = \zero = a_2-a_2 $ for all $a_1,a_2$. Accordingly, we call
a triple ``$\tT$-classical'' \cite[Definition 2.45]{Row16}, when
$a_1(-)a_1 = a_2(-)a_2$ for some $a_1\ne (\pm) a_2$ in $\tT$.

  Examples from classical mathematics
might provide some general intuition about employing $\mathcal A$ to
study $\tT$. A rather trivial example: $\tT$ is the multiplicative
subgroup of a field $\mathcal A$. Or $\mathcal A$ could be a graded
associative algebra, with $\tT$ its multiplicative submonoid of
homogeneous elements. 

But we are more interested in the non-classical situation, involving
semirings which are not rings. Some motivating examples: The
supertropical semiring, where $\tT$ is the set of tangible elements,
the symmetrized semiring, and the power set of a hyperfield, where
$\tT$ is the hyperfield itself. Since hyperfields are so varied,
they provide a good test for this theory. Semirings in general,
without negation  maps, are too broad to yield as decisive results
as we would like, which is the reason that negation maps and triples
are introduced in the first place.

Since we need to correlate two structures ($\tT$ and $\mathcal A$),
as well as the negation map (which could be viewed as a unary
operator), it is convenient to work in the context of universal
algebra, which was designed precisely for the purpose of discussing
diverse structures together. (More recently these have been
generalized to Lawvere's theories and operads, but we do not delve
into these aspects.)

 To round things out, given a triple, we introduce the
\textbf{surpassing relation}~$\preceq$, to replace equality in our
theorems.  (In classical mathematics, $\preceq$ is just equality.)
The quadruple $(\mathcal A, \tT, (-),\preceq)$ is called a
$\tT$-\textbf{system}, cf.~Definition~\ref{sys}.

\subsection{Motivating examples}$ $

 We elaborate the main
non-classical examples motivating this theory.

\subsubsection{Idempotent semirings}$ $

 Tropical geometry
has assumed a prominent position in mathematics because of its
ability to simplify algebraic geometry while not changing certain
invariants (often involving intersection numbers of varieties),
thereby simplifying difficult computations. Outstanding applications
abound, including \cite{AHK,IKS,JP,MikhalkinEnumerative,RSS}.

The main original idea, as expounded in
\cite{IMS,MaclaganSturmfels}, was to take the limit of the logarithm
of the absolute values of the coordinates of an affine variety as
the base of the logarithm goes to $\infty$. The underlying algebraic
structure reverted from $\mathbb C$ to the max-plus algebra $\mathbb
R_{\operatorname{max}}$, an ordered multiplicative monoid in which
one defines $a+b$ to be $\max\{a,b\}.$  This is a \semiring0  and is
clearly additively \textbf{bipotent} in the sense that $a+b \in \{
a, b \}$. Such algebras have been studied extensively some time ago,
 cf.~\cite{Butkovic2003,GaP}.

Idempotent (in particular bipotent) semirings have characteristic 1,
and their geometry has been  studied intensively as
``$F_1$-geometry,'' cf.~\cite{Ber,CC2}. But logarithms cannot be
taken over the complex numbers, and the algebraic structure of
bipotent semirings is often without direct interpretation in
tropical geometry, so attention of tropicalists passed to the field
of Puisseux series, which in characteristic~0 also is an
algebraically closed field, but now with a natural valuation,
thereby making available tools of valuation theory, cf.~\cite{Ber0}.
The collection \cite{BP} presents such a valuation theoretic
approach. Thus one looks for an alternative to the max-plus algebra.

\subsubsection{Supertropical semirings}\label{supert}$ $

  Izhakian \cite{zur05TropicalAlgebra} overcame many of
the structural deficiencies of a max-plus algebra $\tT$ by adjoining
an extra copy of $\tT$, called the \textbf{ghost} copy $\tG$ in
\cite[Definition~3.4]{IR}, as well as $\zero,$ and modifying
addition. More generally, a \textbf{supertropical semiring} is a
semiring with ghosts $\RGnu:= \tT \cup \tGz,$ where $\tGz = \tG \cup
\{ \zero\},$ together with a projection $\nu: R\to \tGz$ satisfying
the extra properties:
\begin{enumerate} \ealph
 \item ($\nu$-Bipotence) $a+b  \in \{a,b\},\ \forall a,b \in R \; \text{such that } \; \nu(a)
\ne \nu(b);$ \pSkip
 \item (Supertropicality) $a+b   =  \nu(a) \quad \text{if}\quad \nu(a) =
\nu(b)$.
\end{enumerate}
The supertropical semiring is \textbf{standard} if $\nu|_\tT$ is
1:1. The supertropical semiring is called the   \textbf{standard
supertropical semifield} when $\tGz $ is a semifield.

Mysteriously, although lacking negation, the supertropical semiring
provides   affine geometry and linear algebra quite parallel to the
classical theory, by taking the negation map $(-)$ to be the
identity, so that $a^\circ = a+a,$ and where the ghost ideal $\tG =
\{ a^\circ: a \in \tT \}$ takes the place of the $\zero$ element. In
every instance, the classical theorem involving equality $f=g$ is
replaced by an assertion that $f = g +\text{ghost},$ called
\textbf{ghost surpassing}. In particular when $g = \zero$ this means
that $f$ itself is a ghost.

For example, an irreducible affine variety should be the set of
points which when evaluated at a given set of polynomials is ghost
(not necessarily $\zero$), leading to:
\begin{itemize}\item
a version of the Nullstellensatz in \cite[Theorem~7.17]{IR},
\item  a link between decomposition of  affine varieties and
(non-unique) factorization of polynomials, illustrated in one
indeterminate in \cite[Remark~8.42 and Theorem~8.46]{IR},
\item a version of the resultant of polynomials that can be computed
by the classical Sylvester matrix and \cite[Theorem~4.12 and
Theorem~4.19]{IzhakianRowen2008Resultants}.
 \end{itemize}

Matrix theory also can be developed along supertropical lines. The
supertropical Cayley-Hamilton theorem \cite[Theorem~5.2]{IR1} says
that the characteristic polynomial evaluated on a matrix is a ghost.
A matrix is called \textbf{singular} when its permanent (the
tropical replacement of determinant) is a ghost;
in~\cite[Theorem~6.5]{IR1} the row rank, column rank, and submatrix
rank of a matrix  (in this sense) are seen to be equal. Solution of
tropical equations is given in \cite{AGG1,AGG2,Ga,IR2}.
Supertropical singularity also gives rise to semigroup versions of
the classical algebraic group SL, as illustrated in \cite{INR}.

Supertropical valuation theory is handled in a series of papers
starting with \cite{IKR1},  also cf.~\cite{Ju2}, generalized further
in~\cite{IzhakianKnebuschRowen2009Refined} and
\cite{IzhakianKnebuschRowen2009Refined1}.

Note that standard supertropical semirings ``almost'' are bipotent,
in the sense that $a_1+a_2\in \{ a_1, a_2 \}$ for any $a_1 \ne a_2$
in $\tT.$ This turns out to be an important feature in triples.

\subsubsection{Hyperfields and other related constructions}$ $

Another algebraic construction is hyperfields \cite{Vi}, which are
multiplicative groups in which sets replace elements when one takes
sums. Hyperfields have received considerable attention recently
\cite{BB,Henry,Ju} in part because of their diversity, and in fact
Viro's ``tropical hyperfield'' matches Izhakian's construction.  But
there are important nontropical hyperfields (such as the hyperfield
of signs, the phase hyperfield, and the ``triangle'' hyperfield)
whose theories we also want to understand along similar lines. In
hyperfield theory, one can replace  ``zero'' by the property that a
given set contains 0.

An intriguing phenomenon is that linear algebra over some classes of
hyperfields follows classical lines as in the supertropical case,
but  the hyperfield of signs provides easy counterexamples to
others, as discussed in \cite{AGR}.

\subsubsection{Fuzzy rings} $ $

Dress \cite{Dr} introduced ``fuzzy rings'' a while ago in connection
with matroids, and these also have been seen recently to be related
to hypergroups in \cite{BB,DW0,DW,GJL,MacR}.

\subsubsection{Symmetrization}\label{symm}$ $

This construction  uses   Gaubert's ``symmetrized algebras''
\cite{Ga,Pl,AGG1,GaP} (which he designed for linear algebra) as a
prototype. We start with $\tT$, take $\mathcal A = \widehat{\tT}: =
\tT \times \tT$, and define the ``switch map'' $(-)$ by
$(-)(a_0,a_1) = (a_1,a_0)$. The reader might already recognize this
as the first step in constructing the integers from the natural
numbers, where one identifies $(a_0,a_1)$ with $(a_0',a_1')$ if
$a_0+a_1' = a_0'+a_1,$ but the trick here is to recognize the
equivalence relation without modding it out, since everything could
degenerate in the nonclassical applications. Equality $(a_0,a_1) =
(b_0,b_1)$ often is replaced by the assertion $(a_0,a_1) = (b_0,b_1)
+(c,c)$ for some $c \in \tT.$ The ``symmetrized'' $\tT$-module
  also can be viewed as a {$\widehat \tT$}-super-module
(i.e., 2-graded), via the \textbf{twist action}
 \begin{equation}\label{twi} (a_0,a_1)\ctw (b_0,b_1) =
 (a_0b_0 + a_1 b_1, a_0 b_1 + a_1 b_0),  \end{equation}
utilized in \cite{JoM} to define and study the prime spectrum.

\subsubsection{Functions}\label{func}$ $

Given some  structure $\mathcal A$ and a set $S$, we can define the
set of functions $\Fun (S, \mathcal A)$ from $S$ to  $ \mathcal A$,
with operators defined elementwise, i.e., $$\omega (f_1, \dots
,f_m))(s) = (\omega (f_1(s), \dots ,f_m(s)).$$ For example, taking
$S = \{ 1, \dots, n\},$ a polynomial $f(\la_1, \dots, \la _n) $ can
be viewed as a function from $\Fun (S, \mathcal A)$ to $ \mathcal
A$. In other words, we take the substitution $\la _i \mapsto a_i$
and then send $(a_1, \dots, a_n)$ to $f(a_1, \dots, a_n).$

One  often identifies polynomials in terms of their values as
functions on their set of definition. Then $\la^2$ and $\la$ would
be identified as polynomials over the finite field $\mathbf F _2$.

\subsection{Negation maps, triples, and systems}\label{trsys}$ $

These varied examples and their theories, which often mimic
classical algebra, lead one to wonder whether the parallels among
them are happenstance, or whether there is some straightforward
axiomatic framework within which they can all be gathered and
simplified.   Unfortunately  semirings may lack negation, so we also
implement a formal negation map~$(-)$ to serve as a partial
replacement for negation.

\begin{defn}\label{negmap}
 A \textbf{negation map} on a $\tT$-module $(\mathcal
A,\cdot,+)$ is  a map $(-) :(\tT,+) \to (\tT,+)$ together with a
 semigroup isomorphism
$$(-) :(\mathcal A,+) \to (\mathcal A,+),$$ both of order $\le 2,$  written
$a\mapsto (-)a$,   satisfying
 \begin{equation}\label{neg} (-)(ab)  = ((-)a) b = a((-)b), \quad \forall a\in \tT,\ b \in \mathcal A.\end{equation}

\end{defn}

Obvious examples of negation maps are the identity map, which might
seem trivial but in fact is the one used in supertropical algebra,
the switch map ($(-)(a_0,a_1) = (a_1,a_0))$ in the symmetrized
algebra,
 the usual negation map $(-) a = -a$ in classical algebra, and the hypernegation in the definition of hypergroups.
Accordingly, we say that the negation map $(-)$ is of the
\textbf{first kind} if $(-)a =a$ for all $a \in \tT,$ and of the
\textbf{second kind} if $(-)a \ne a$ for all $a \in \tT.$

As indicated earlier, the quasi-zeros
 take the role customarily
assigned to the zero element. In the supertropical theory the
 quasi-zeros are the ``ghost'' elements.
 In \cite[Definition~2.6]{AGG2} the quasi-zeros are called ``balanced
elements'' and have the form $(a,a)$.

 When $\one \in \tT$, the element $(-)\one$ determines the negation map, since
 $(-)b = (-)(\one b) = ((-)\one)b.$ When $\tT \subseteq \mathcal A$, several
important elements  of $\mathcal A$ then are:

 \begin{equation}\label{eeq} e = \one ^\circ = \one (-) \one, \quad e' = e + \one, \quad e^\circ = e (-) e = e+ e = 2e.\end{equation}

   The most important quasi-zero for us is $e.$ (For fuzzy rings,
$e = \one + \epsilon.$) But $e$ need not absorb in multiplication;
  rather, in any \semiring0 with negation, Definition~\ref{modu13}  implies
\begin{equation}\label{emul} a e=  a (-) a = a ^\circ .\end{equation}

%
%

 \begin{defn}\label{sursys0}
A \textbf{pseudo-triple} is a collection $(\mathcal A, \tT, (-)),$
where $\mathcal A$ is a $\tT $-module with $\tT \subset \mathcal A$,
and $(-)$ is a negation map.
 A
$\tT$-\textbf{pseudo-triple} is a pseudo-triple in which $\tT
\subseteq \mathcal A$, where the negation map on $\mathcal A$
restricts to the negation map on $\tT$.

 A
$\tT$-\textbf{triple} is a $\tT$-pseudo-triple, in which $\tT \cap
\mathcal A^ \circ = \emptyset$ and $\tT$
generates $( \mathcal A,+).$  
\end{defn}

\begin{example}\label{neg217}
 The  main non-classical  examples are:

\begin{itemize}\item (The standard supertropical triple) $(\mathcal A, \tT, (-))$ where  $\mathcal A = \tT \cup
\tG$ as before and $(-)$ is the identity map.

\item (The symmetrized triple) $(\hat {\mathcal
A},\widehat {\tT},(-))$ where  $\hat {\mathcal A} = \mathcal A
\times \mathcal A$ with componentwise addition, and $\widehat {\tT}
=(\tT \times \{ \zero \}) \cup ( \{ \zero \} \times \tT)$ with
multiplication $\widehat {\tT} \times \widehat {\mathcal A}\to
\widehat {\mathcal A}$ given by
$$(a_0,a_1)(b_0,b_1) = (a_0 b_0 + a_1 b_1, a_0 b_1 + a_0 b_1).$$
Here we take $(-)$ to be the switch map ($(-)(a_0,a_1) =
(a_1,a_0)),$ which is of second kind.

\item (The hyperfield pseudo-triple  \cite[\S2.4.1]{Row16}) $(\mathcal P (\tT),\tT,(-))$ where $\tT$ is the original hyperfield, $\mathcal P (\tT)$
is its power set (with componentwise operations), and $(-)$ on the
 power set is induced from the hypernegation. Here $\preceq$ is
$\subseteq$.

\item (The fuzzy triple \cite[Appendix A]{Row16}) For any $\tT$-monoid module  $\mathcal A$ with  an
element $\one'\in \tT$ satisfying $(\one')^2 =\one,$
 we can define a negation map $(-)$ on $\tT$ and
$\mathcal A$ given by $a \mapsto \one' a.$
 In particular, $(-)\one = \one'$.

\item (The polynomial triple) Since any polynomial is a finite sum
of monomials, we can take any triple  $(\mathcal A, \tT, (-))$ and
form $(\mathcal A [\Lambda], \tT_{A [\Lambda]}, (-))$ where  $\tT_{A
[\Lambda]}$ is the set of monomials. Negation is taken elementwise.
This can all be done formally, but from a geometric perspective it
is useful to view polynomials as functions on varieties.
\end{itemize}
\end{example}

Although we introduced pseudo-triples since $\tT$ need not generate
$(\mathcal A,+)$ (for example, taking $\mathcal A = \mathcal P
(\tT)$ for the phase hypergroup), we are more concerned with
triples, and furthermore in a pseudo-triple one can take the
sub-triple generated by $\tT$. More triples related to tropical
algebra are presented in \cite[\S3.2]{Row16}.

Structures other than monoids also are amenable to such an approach.
This can all be formulated axiomatically in the context of universal
algebra, as treated for example in \cite{Jac1980}. Once the natural
categorical setting is established, it provides the context in which
tropicalization (described below) becomes a functor, thereby
providing guidance to understand tropical versions of an assortment
of mathematical structures.

 \begin{defn}\label{sys}
 Our structure of choice, a  $\tT$-\textbf{system},
is a quadruple $(\mathcal A, \tT, (-), \preceq),$ where $(\mathcal
A, \tT, (-))$ is a $\tT$-triple and $\preceq$ is a
 ``$\tT$-surpassing relation (\cite[Definition 2.65]{Row16}),''
satisfying the crucial property that if $a + b \succeq \zero$ for
$a,b \in \tT$ then $b= (-)a $.

 The main
  $\tT$-surpassing relations are:

\begin{itemize} \item (for supertropical, symmetrized, and fuzzy rings) $\preceq_\circ$, defined by $a \preceq_\circ b$
if $b = a + c^\circ$ for some $c$.
\item (on sets) $\preceq$ is $\subseteq$.
\end{itemize}
\end{defn}

The relation $\preceq$ has an important theoretical role, replacing
``$=$'' and enabling us to define a broader category than one would
obtain directly from universal algebra, cf.~\cite[\S6]{Row16}. One
major reason why $\preceq$ can formally replace equality in much of
the theory is found in the ``transfer principle'' of \cite{AGG2},
given in the context of systems in  \cite[Theorem~6.17]{Row16}.

\subsubsection{Ground triples versus module triples}$ $

 Classical  structure theory involves
the investigation of an algebraic structure as a small category (for
example, viewing a monoid as a category with a single object whose
morphisms are its elements), and homomorphisms then are functors
between   two of these small categories. On the other hand, one
obtains classical representation theory via an abelian category,
such as the class of modules over a given ring.

Analogously, there are two aspects of triples. We call a triple
(resp.~system) a \textbf{ground triple (resp.~ground system)} when
we study it as a small category with a single object in its own
right, usually a semidomain. Ground triples have the same flavor as
Lorscheid's blueprints (albeit slightly more general, and with a
negation map), whereas representation theory leads us to
\textbf{module systems}, described below in \S\ref{modsy} and
\S\ref{JuR22}.

This situation leads to a fork in the road: The first path takes us
to a structure theory based on functors of ground systems,
translating into homomorphic images of systems via congruences in
\cite[\S 6]{JuR1} (especially \textbf{prime systems}, in which the
product of non-trivial congruences is nontrivial \cite[\S
6.2]{JuR1}). Ground systems often are designated in terms of the
structure of $\mathcal A$ or $\tT$, such as ``semiring systems''  or
``Hopf systems'' or ``hyperfield systems.''

The paper \cite{AGR} has a different flavor, dealing with matrices
and linear algebra over ground systems, and focusing on subtleties
concerning Cramer's rule and the equality of row rank, column rank,
and submatrix rank.

  The second path takes us to categories of module
systems. In \cite{JuR1} we also bring in tensor products and Hom
functors.  In~\cite{JuMR2}  we develop the  homological theory,
relying on work done already by Grandis
\cite{grandis2013homological} under the name of
\textbf{$N$-category} and \textbf{homological category} (without the
negation map), and there is a parallel approach of Connes and
Consani in \cite{CC2}.

\section{Contents of \cite{Row16}: meta-tangible systems}

 The emphasis in~\cite{Row16} is on ground systems.
One can apply the familiar constructions and concepts of classical
algebra (direct sums {\cite[Definition 2.10]{Row16}}, matrices
\cite[\S4.5]{Row16}, involutions \cite[\S4.6]{Row16}, polynomials
\cite[\S4.7]{Row16}, localization \cite[\S4.8]{Row16}, and tensor
products {\cite[Remark~6.34]{Row16}}) to produce new triples and
systems. The simple tensors $a\otimes b$ where $a,b\in \tT$ comprise
the tangible elements of the tensor product. The properties of
tensors and Hom are treated in much greater depth in \cite{JuR1}.

\subsubsection{Basic properties of triples and systems}\label{modsy1}$
$

Let us turn to important properties which could hold in triples. One
basic axiom for this theory, holding in all tropical situations and
many related theories, is:

\begin{defn}\label{metatan} A uniquely negated $\tT$-triple $(\mathcal A, \tT, (-))$ is \textbf{meta-tangible}, if
the sum of two tangible elements is tangible unless they are
quasi-negatives of each other.

 A special case:  $(\mathcal A, \tT, (-))$ is  $(-)$-\textbf{bipotent} if $a + b \in \{a ,b\}$
whenever $a, b \in \tT$ with $b \neq (-) a.$ In other words, $a + b
\in \{a ,b, a^\circ \}$ for all $a,b \in \tT$.
 \end{defn}

The stipulation in the definition that $b\ne (-)a$ is of utmost
importance, since otherwise all of our main examples would fail.
\cite[Proposition~5.64]{Row16} shows how to view a meta-tangible
triple as a hypergroup, thereby enhancing the motivation of
transferring hyperfield notions to triples and systems in general.

 Any meta-tangible triple satisfying $\tT \cap \mathcal A^\circ =
\emptyset$ is uniquely negated. The supertropical triple is
$(-)$-bipotent, as is the modification of the symmetrized triple
described in~\cite[Example~2.53]{Row16}.

  The Krasner hyperfield triple (which is just the supertropicalization of the Boolean semifield $\mathbb B$) and the
triple arising from the hyperfield of signs (which is just the
symmetrization of $\mathbb B$) are $(-)$-bipotent, but the phase
hyperfield triple and the triangle hyperfield triple are not even
metatangible (although the latter is idempotent). But as seen in
{\cite[Theorem 4.46]{Row16}}, hyperfield triples satisfy another
different property of independent interest:

\begin{defn}{\cite[Definition 4.13]{Row16}}\label{precedeq}
A surpassing relation  $\preceq$ in a system is called
$\tT$-\textbf{reversible} if
 $a \preceq  b +c   $ implies  $b \preceq a (-) c$  for $a,b \in \tT.$
\end{defn}

The category of hyperfields as given in  \cite{Ju} can be embedded
into the category
 of uniquely negated $\tT$-reversible systems  (\cite[Theorem
 6.7]{Row16}). Reversibility enables one to apply systems to matroid
theory, although we have not yet embarked on that endeavor in
earnest.

 The \textbf{height} of an element $c \in \mathcal A$ (sometimes
called ``width'' in the group-theoretic literature) is the
minimal~$t$ such that $c = \sum _{i=1}^t a_i$ with each $a_i \in
\tT.$ (We say that $\zero$ has height~ 0.) By definition, every
element of a triple has finite height. The \textbf{height}
of~$\mathcal A$ is the maximal height of its elements, when these
heights are bounded. For example, the supertropical semiring has
height 2, as does the symmetrized semiring of an idempotent
semifield $\tT$.

 Some unexpected examples of meta-tangible systems sneak in
when the triple has height $\ge 3$, as described in
{\cite[Examples~2.73]{Row16}}.

In \cite[\S5]{Row16} the case is presented that meta-tangibility
could be the major axiom in the theory of ground systems over a
group $\tT$, leading to a bevy of structure theorems on
meta-tangible systems, starting with the observation \cite[Lemma
5.5]{Row16} that for $a,b \in \tT$ either
 $ a = (-)b$, $a+b = a$ (and thus $a^\circ + b = a^\circ$), or $a^\circ +  b
= b.$

In \cite[Proposition 5.17]{Row16} the following assertions are seen
to be
 equivalent for a  triple $(\mathcal A, \tT,
(-))$  containing $\one$:
\begin{enumerate}\eroman\item $\tT \cup  \tT^\circ = \mathcal A,$
\item $\mathcal A$ is  meta-tangible  of height $\le 2$,
\item  $\mathcal A$ is  meta-tangible with $e' \in \{ \one, e\}.$
 \end{enumerate}

  We obtain the following results for a meta-tangible system $(\mathcal A, \tT,
(-), \preceq)$:

\begin{itemize}\item \cite[Theorem 5.20]{Row16} If $\mathcal A$ is not
$(-)$-bipotent then $e'= \one,$ with $\mathcal A$ of characteristic
2 when $(-)$ is of the first kind.
\item \cite[Theorem 5.27]{Row16}  Every element has the form $c^\circ$ or $m c$ for $c
\in \tT$ and $m \ne 2$.
 The extent to which this presentation is unique is described in \cite[Theorem
5.31]{Row16}.
\item \cite[Theorem 5.33]{Row16} Distributivity follows from the other
axioms.
\item \cite[Theorem 5.34]{Row16} The surpassing relation $\preceq$
must ``almost'' be   $\preceq_\circ$.
\item \cite[Theorem 5.37]{Row16} A key property of fuzzy
rings holds.
\item \cite[Theorem 5.43]{Row16} Reversibility holds, except in one pathological situation.
\item \cite[Theorem 5.55]{Row16} A criterion is given in terms of sums of squares for  $(\mathcal A, \tT,
(-))$ to be isomorphic to a symmetrized triple.
\end{itemize}

One would want a classification theorem of meta-tangible systems
that reduces  to    classical algebras, the standard supertropical
semiring, the symmetrized semiring,   layered semirings, power sets
of various hyperfields, or fuzzy rings, but there are several
exceptional cases. Nonetheless, \cite[Theorem 5.56]{Row16} comes
close. Namely, if $(-)$ is of the first kind then either $\mathcal
A$ has characteristic 2 and height 2 with $\mathcal A^\circ$
bipotent, or $(\mathcal A, \tT, (-))$ is isomorphic to a layered
system. If $(-)$ is of the second kind then either  $\tT$ is
$(-)$-bipotent, with $\mathcal A$ of height 2 (except for an
exceptional case), $\mathcal A$ is isometric to a symmetrized
semiring when $\mathcal A$ is real, or $\mathcal A$ is classical.
More information about the exceptions are given in
\cite[Remark~5.57]{Row16}.

\cite[\S 7]{Row16} continues with some rudiments of linear algebra
over a ground triple, to be discussed shortly.

In \cite[\S 8]{Row16}, tropicalization is cast in terms of a functor
on systems (from the classical to the nonclassical). This principle
enables one in \cite[\S 9]{Row16} to define the ``right'' tropical
versions of classical algebraic structures, including exterior
algebras, Lie algebras, Lie superalgebras, and Poisson algebras. 

\section{Contents of \cite{AGR}: Linear algebra over systems}

The paper \cite{AGR} was written with the objective of understanding
some of the diverse theorems in linear algebra    over semiring
systems.
 We  define a set of
  vectors $\{v_i \in
\mathcal A ^{(n)}: i \in I\}$ to be \textbf{$\tT$-dependent} if
$\sum _{i \in I'} \a_i
v_i \in (\mathcal A^\circ)^{(n)}$ for some nonempty subset $I' \subseteq I$ and $\a_i \in \tT$, and the 
\textbf{row rank} of a matrix to be the maximal number of
$\tT$-independent rows.

A \textbf{tangible vector} is a vector all of whose elements are
in~$\tTz .$ A \textbf{tangible matrix} is a matrix all of whose rows
are tangible vectors.

 The
\textbf{$(-)$-determinant} $\absl A $ of an $n \times n$ matrix $ A
= a_{i,j}  $  is
 \begin{equation*}\label{eq:tropicalDetsign}
  \sum_{\pi \in  S_n}  (-)^{\pi} a_\pi,
\end{equation*}
where $a_\pi:= \prod_{i=1}^n a_{i,\pi(i)}$.

\begin{defn}\label{adjo} Write $a_{i,j}'$ for
the $(-)$-determinant of the $j,i$ minor of a matrix $A$. The
\textbf{$(-)$-adjoint} matrix $\adj A$ is $(a_{i,j}')$.
\end{defn}

\begin{thm}[{Generalized Laplace identity, \cite[Theorem~1.56]{AGR}}]\label{Lap}

Suppose we fix $I = \{ i_1, \dots, i_m \}\subset \{1, 2, ... , n\},$
and for any set $J \subset \{1, 2, ... , n\},$ with $|J| = m,$ write
$(a_{I,J})$ for the $m\times m$ minor $(a_{i,j}: i\in I, j \in J)$
and  $(a_{I,J}')$ for the $(-)$-determinant of the $(n-m)\times
(n-m)$ minor obtained by deleting all rows from $I$ and all columns
from $J$. For $J = \{ j_1, \dots, j_m\}$ write $(-)^J$ for $(-)^{
j_1 +\dots + j_m}.$ Then $|A| = \sum _{J : |J|= m} (-)^{I}(-)^{J}
a'_{I,J}|a_{I, J}|.$
\end{thm}
\begin{proof} One can copy the proof of \cite[Theorem 1, \S~2.4]{matbook}. Namely, there are $ \binom n {m}$ summands, each of
which contribute $m!(n-m)!$ different terms to \eqref
{eq:tropicalDetsign} with the proper sign, so altogether we get
$m!(n-m)!\binom n {m} = n!$ terms from \eqref {eq:tropicalDetsign},
in other words all the terms.
\end{proof}

A matrix $A$ is \textbf{nonsingular} if $|A|\in \tT.$
 Vectors are defined to be \textbf{independent} if and only if no tangible linear
combination is in ${\mathcal A^\circ}^{(n)}$. The \textbf{row rank}
of $ A $ is the largest number of independent rows of $A$.

 The \textbf{submatrix
rank} of $ A $ is the largest size of a nonsingular square submatrix
 of $ A $.


In \cite[Corollary 7.4]{Row16} we see that the submatrix rank of a
matrix over a $\tT$-cancellative meta-tangible triple is less than
or equal to both the row rank and the column rank. This is improved
in \cite{AGR}:

\begin{thm}[{\cite[Theorem 4.7(i)]{AGR}}]\label{A1th}
 Let $(\mathcal A,
\tT, (-), \preceq_{\circ})$  be a system. For any vector $v$, the
vector $y =
  \frac 1 {|A|}{\adj Av }$ satisfies $|A|v \preceq_{\circ} A y.$ In particular,
if $|A|$ is invertible in $\tT$, then $x:= \frac 1 {|A|}{\adj Av }$
satisfies $v \preceq_{\circ} A x.$
\end{thm}

The existence of a tangible such $x$ is subtler.  One
  considers valuations of systems
 $\nu : (\mathcal
A,+) \to (\tG,+)$,   and their \textbf{fibers} $\{ a \in \tT: \nu
(a) = g\}$ for $g \in \tG;$ we call the system
$\tT$-\textbf{Noetherian} if any  ascending chain of fibers
stabilizes.
{
\begin{thm}[{\cite[Corollary~4.26]{AGR}}] In a $\tT$-Noetherian
$\tG$-valued system $\mathcal A$, if  $|A|$ is invertible, then for
any vector $v$, there is a tangible vector~$x$ with $|x| = \frac 1
{|A|}{|\adj A|v }$, such that $Ax + v\in \mathcal A^{\circ}.$
\end{thm}
 One  obtains uniqueness of $x$ {\cite[Theorem 4.7(ii)]{AGR}} using
 a property called ``strong balance elimination.''
  After
translating some more of the concepts of \cite{AGG2} into the
language of systems, we turn to the question, raised privately for
hyperfields  by Baker:

\bigskip
 {\bf Question A.} When does the submatrix rank equal  the row
rank? \medskip

Our initial hope was that this would always be the case, in analogy
to the supertropical situation.
 However,   Gaubert
observed that a (nonsquare) counterexample to Question A already can
be found in~\cite{AGG1}, and the underlying system even is
meta-tangible.  Here the kind of negation map is critical: A~rather
general counterexample for triples of the second kind is given in
\cite[Proposition~3.3]{AGR}; the essence of the example already
exists in the ``sign hyperfield.''   Although the counterexample as
given is a nonsquare ($3\times 4$) matrix, it can be modified to an
$n\times n$  matrix for any~$n \ge 4$. This counterexample is
minimal in the sense that Question~A has a positive answer for
 $n\le 2$ and for $3 \times 3$ matrices under a mild assumption, cf.~\cite[Theorem~5.7]{AGR}.

Nevertheless, positive results  are available. A~positive answer for
Question A along the lines of \cite[Theorems~5.11, 5.20, 6.9]{AGG2}
is given in \cite[Theorem~5.8]{AGR} for systems satisfying certain
technical conditions. In \cite[Theorem~5.11]{AGR}, we show that
Question~A has a positive answer for square matrices over
meta-tangible triples of first kind  of height~2, and this seems to
be the ``correct'' framework in which we can lift theorems from
classical algebra.  A~positive answer for all rectangular matrices
is given in \cite[Theorem~5.19]{AGR}, but with
 restrictive hypotheses that essentially reduce to the
supertropical situation.

\section{Contents of \cite{GatR}: Grassmann semialgebras}\label{Grass}

This paper unifies classical and tropical theory. Ironically, even
though a vector space over a \semifield0 need not contain a negation
map, the elements of degree $\ge 2$ does have a negation map given
by $(-)(v\otimes w) = w \otimes v,$ so these can be viewed  in terms
of triples and systems. In the process we investigate Hasse-Schmidt
derivations on Grassmann {exterior} systems and use these results to
provide a generalization of  the Cayley-Hamilton theorem in
\cite[Theorem~3.18]{GatR}.

  But the version given in
\cite[Theorem~2.6 and Definition~2.12]{GatR} (over a free module $V$
over an arbitrary semifield) is the construction which seems to
``work.'' This is obtained by taking a given base $\{b_0, b_1,
\dots, b_{n-1}\}$ of $V$,
 defining $(-)b_i \w b_j = b_j \w b_i$ and $b_i \w b_i = \zero$ for each $0 \le i < j \le n-1,$ and
 extending  to all of the tensor algebra $T(V)$. This does not imply $v \w v =
 \zero$ for arbitrary $v \in V$; for example, taking $v = b_0 + b_1$
 yields $v \w v = b_0 b_1 + b_1 b_0$ which need not be $\zero,$ but
 it acts like $\zero$.

Thus the Grassmann semialgebra of a free module $V$ has a natural
negation map  on all homogeneous vectors,  with the ironic exception
of $V$ itself, obtained by switching two tensor components.
 This
provides us ``enough'' negation, coupled with the relation
$\preceq_\circ$ to carry out the theory (But we do not mod out all
elements $v \otimes v$ for arbitrary $v \in V!$.)

 Our main theorem,
\cite[Theorem~3.17]{GatR},
 describes the relation between a ``Hasse-Schmidt
derivation'' $\Dz$ and its ``quasi-inverse'' $\overline {D} \{z\}$
defined in such a way to yield:

\begin{thm}
 $\overline {D} \{z\} (\Dz  u\w v)\succeq u\w \overline {D} \{z\}
v.$
 \end{thm}

In the classical case, one gets equality as shown in
\cite[Remark~3.19]{GatR}, so we recover \cite{GaSc2}.  
 Our main application is a
generalization of the Cayley-Hamilton theorem to semi-algebras.

\begin{thm}[{\cite[Theorem~3.17]{GatR}}]
$
\left((D_nu+e_1D_{n-1}u+\cdots+e_{n}u)\w v\right)
(-)\left((D_nu+e'_1D_{n-1}u+\cdots+e'_{n}u)\w v \right) \succeq 0
 $ for all $u\in \bw^{>0}V_n$,
 which we also relate to
 super-semialgebras.
 \end{thm}

\section{Contents of \cite{JuR1}: Basic categorical considerations}

The paper \cite{JuR1} elaborates on the categorical aspects of
systems, with emphasis on important functors. The convolution triple
$\mathcal C^S$ \cite[Proposition 3.7]{JuR1} embraces important
constructions including the symmetrized triple and polynomial
triples (via the convolution product given before \cite[Definition
3.4]{JuR1}).

In order to pave the way towards geometry, we consider ``prime''
systems and congruences, proving some basic results about polynomial
systems:

\begin{thm}[{\cite[Proposition 6.19]{JuR1}}]
 For every $\tT$-congruence $\Cong$ on a commutative $\tT$-semiring system,  $\sqrt{\Cong}$
  is an intersection  of prime $\tT$-congruences.
 \end{thm}

   \medskip

 \begin{thm}[{\cite[Theorem 6.29]{JuR1}}]  Over a commutative prime triple $\mathcal A$,
any nonzero polynomial $f \in \tT[\la]$ of degree $n$ cannot have
$n+1$ distinct $\circ$-roots in $\tT$. \end{thm}

  \begin{thm}[{\cite[Corollary 6.30]{JuR1}}]
 If $(\mathcal A,\tT,(-))$ is a  prime commutative
 triple with $\tT$ infinite, then so is  $(\mathcal A[\la],\tT,(-)).$
 \end{thm}

As in \cite{Lor1,Lor2}, the emphasis on \cite{JuR1} is for $\tT$ to
be a cancellative multiplicative monoid (even a group), which
encompasses many major applications. This slights the Lie theory,
and indeed one could consider  Hopf systems. Motivation can be found
in~\cite{So}.

An issue that  must be confronted   is the proper definition of
morphism, cf.~\cite[Definitions 4.1, 4.3]{JuR1}. In categories
arising  from universal algebra, one's intuition would be to take
the homomorphisms, i.e., those maps which preserve equality in the
operators. We call these morphisms ``strict.'' However,  this
approach loses some major examples of hypergroups. Applications in
tropical mathematics and hypergroups (cf.~\cite[Definition 2.3]{Ju})
tend to depend on the ``surpassing relation'' $\preceq$
(\cite[Definition~2.65]{Row16}), so we are led to a broader
definition called \textbf{$\preceq$-category} in \cite[\S 4,
Definition~4.34]{JuR1}. $\preceq$-morphisms
    often provide the correct venue for
studying ground systems.
 On the other hand, \cite[Proposition~4.37ff.]{JuR1} gives a way of
 verifying that some morphisms automatically are strict.


The situation is stricter for module systems  \cite[\S 5, 7,
8]{JuR1} over  ground triples. The sticky point here is that the
semigroups of morphisms $\Mor(A,B)$ in our $\tT$-module categories
are not necessarily groups, so the traditional notion of abelian
category has to be replaced by ``semi-abelian,'' \cite[Theorem
7.16]{JuR1}, and these  lack some of the fundamental properties of
abelian categories. The tensor product is only functorial when we
restrict our attention to the stricter definition of morphisms,
\cite[Proposition 5.8]{JuR1}!

\subsection{Module systems}\label{modsy}$ $

 In both cases, in the theory of semirings and their modules,
homomorphisms are described in terms of congruences, so congruences
should be a focus of the theory. The null congruences contain the
diagonal, and not necessarily zero, and lead us to null morphisms,
\cite[Definition 3.2]{JuR1}. An alternate way of viewing congruences
in terms of ``transitive'' modules of $\widehat{\mathcal M}$ is
given in  \cite[\S 8.4]{JuR1}. In \cite[\S 9]{JuR1}, ``Hom'' is
studied together with its dual, and again one only gets all the
desired categorical properties (such as the adjoint isomorphism
\cite[Lemma~9.8]{JuR1}) when considering strict morphisms. In this
way, the categories comprised of strict morphisms should be amenable
to a categorical view,  to be carried out in  \cite{JuMR2} for
homology, again at times with a Hopfian flavor.

The functors between the various categories arising in this theory
are described in \cite[\S 10]{JuR1}, also with an eye towards
valuations of triples.

%
%
%
%
%
%
%

 As in classical algebra, the ``prime'' systems
\cite[Definitions~2.11, 2.25]{Row16} play an important role in
affine geometry, via the  Zariski topology \cite[\S5.3.1]{JuR1}, so
it is significant that we have a version of the fundamental theorem
of algebra in \cite[Theorem~7.29]{JuR1}, which implies that a
polynomial system over a prime system is prime
 \cite[Corollary~7.30]{JuR1}.

\section{Contents of \cite{JuMR1}: Projective module systems}\label{JuR22}

 We start  with
 a $\preceq$ version of split epics (weaker than the classical
definition):

\begin{defn}\label{split8}
An epic $\pi: \mathcal M\to \mathcal N$   \textbf{$\preceq$-splits}
if there is an $N$-monic $\nu:\mathcal  N \to \mathcal  M$ such that
$ \one_{\mathcal  N} \preceq \pi \nu  $. In this case, we also say
that $(\pi ,\nu)$ \textbf{$\preceq$-splits}, and $\mathcal N$ is a
$\preceq$-\textbf{retract} of $\mathcal M$.

 A module system
$\mathcal M = (\mathcal M, \tT_{\mathcal M}, (-),\preceq) $ is the
$\tT$-\textbf{$\preceq$-direct sum} of subsystems $(\mathcal M_1,
\tT_{\mathcal M_1}, (-),\preceq)$ and~$(\mathcal M_2, \tT_{\mathcal
M_2}, (-),\preceq)$ if $\tT_{\mathcal M_1} \cap \tT_{\mathcal M_2} =
\emptyset$ and every $a \in \mathcal M$ can be written   $a \preceq
a_1 + a_2$ for $a_i \in \mathcal M_i$.
\end{defn}

This leads to projective module systems.

\begin{defn} 

  A $\tT$-module system $\mathcal P$ is   \textbf{projective} if for any strict epic  $h: \mathcal M \to \mathcal
  M'$ of $\tT$-module systems,
 every  morphism $ f: \mathcal P \to \mathcal M'$ lifts to a
 morphism $\tilde f: \mathcal P \to\mathcal M$,
 in the sense that $h\tilde f =  f.$
%

 $\mathcal P$ is  $\preceq$-\textbf{projective} if for any ($\tT$-module system) strict epic  $h: \mathcal M \to \mathcal M',$
 every morphism $ f: \mathcal P \to \mathcal M'$ $\preceq$-lifts to a
 morphism $\tilde f: \mathcal P \to \mathcal M$,
 in the sense that $f \preceq h\tilde f  .$
\end{defn}

Their fundamental properties are then obtained, including the
$\preceq$-Dual Basis Lemma \cite[Proposition~5.14]{JuMR1}, leading
to $\preceq$-projective
 resolutions and
 $\preceq$-projective
 dimension.

Building on projective modules, homology is work in progress
\cite{JuMR2}, as is parallel work in geometry. One obtains a
homology theory in the context of homological categories
\cite{grandis2013homological} and derived functors, in connection to
the recent work of Connes and Consani \cite{CC2}.

\section{Interface between systems and tropical mathematics}\label{comp}

We conclude by relating systems to other approaches taken in
tropical mathematics, \cite{IMS,MaclaganSturmfels,Ber0,Ber}.

\subsection{Tropical versus supertropical}\label{supert1}$ $

First we consider briefly some of the basic tools in affine tropical
geometry, to see how they relate to the supertropical setting.

\subsubsection{The ``standard'' tropical approach}$ $

 One often works in the polynomial semiring $\mathbb
R_{\operatorname{max}}[\Lambda]$, although  here we replace $\mathbb
R_{\operatorname{max}}$ by any  ordered semigroup $(\Gamma,\cdot)$,
with $\Gz : = \Gamma \cup \{ \zero \}$ where $\zero a = \zero$ for
all $a \in \Gamma$.   For $\mathbf i = (i_1, \dots, i_n)\in \mathbb
N^{(n)}$, we write $\Lambda ^\mathbf i $ for $ \la_1^{i_1}\cdots
\la_n^{i_n}.$ A tropical hypersurface of a tropical polynomial $f =
\sum _{\mathbf i } \a_{\mathbf i} \Lambda ^\mathbf i  \in
\Gamma[\Lambda]$ is defined as the set of points in which two
monomials take on the same dominant value, which is the same thing
as the supertropical value of $f$ being a ghost.

\begin{defn}[{\cite[Definition~5.1.1]{GG}}] Given a   polynomial $f =
\sum \a_{\mathbf i}\Lambda ^\mathbf i ,$  define  $\supp( f )$ to be
all the tuples $\mathbf i =({i_1}\cdots  {i_n})$  for which the
monomial in $f$ has nonzero coefficient $\a_{\mathbf i}$, and for
any such monomial~$h,$ write $f_{\hat h}$ for the polynomial
obtained from deleting $h$ from the summation.

 The \textbf{bend relation} of $f$ (with respect to a tropical hypersurface $V$) is generated by all $$\{f  \equiv_{\operatorname{bend}} f_{\hat
 h=  \a_{\mathbf i}\Lambda ^\mathbf i }:\
\mathbf i\in  \supp( f )\}.$$
\end{defn}

The point of this definition is that the variety $V$ defined by a
tropical polynomial is defined by two monomials (not necessarily the
same throughout) taking equal dominant values at each point of $V$,
and then the bend relation reflects the equality of these values on
$V$, thence the relation.

\subsubsection{Tropicalization and tropical ideals}\label{trop1}$ $

Finally, one needs to relate tropical algebra to Puiseux series via
the following tropicalization map.

\begin{defn} For  any additive group $\mathcal M$, one can define the group
$\mathcal M\{\{t\}\} $ of Puiseux series on the variable $t$, which
is the set of formal series of the form $ p = \sum_{k =
\ell}^{\infty} c_k t^{k/N}$ where $N \in \mathbb{N}$, $\ell \in
\mathbb{Z}$, and $c_k \in \mathcal A$.\end{defn}

Customarily, $\mathcal M = \mathbb C$.

\begin{rem} If $\mathcal{M}$ is an algebra, then $
\mathcal{M}\{\{t\}\} $ is also an algebra under the usual
convolution product.
\end{rem}

\begin{defn} One has the \textbf{Puiseux
valuation} $val : \widehat{\mathcal{M}\{\{t\}\} } \setminus \{0\}
\rightarrow \Gamma$ defined by
\begin{equation}
val(p) = \min_{c_k \neq 0}\{k/N\},
\end{equation}
 sending a Puiseux series to its value under the Puiseux valuation
$v$. This induces a map $$\tropa: \mathcal{M}\{\{t\}\}[\Lambda] \to
\Gamma[\Lambda] \},$$ called \textbf{tropicalization}, sending
$p(\la_1 \cdots \l_n) : = \sum p_{\mathbf i}\la_1 \cdots \l_n$ to
$\sum \val(p_{\mathbf i})\la_1 \cdots \l_n$.

Suppose $I \triangleleft  \mathcal{M}\{\{t\}\}[\Lambda].$ The bend
congruence on $\{ \tropa(f): f \in I \}$ is denoted as $\Trop(I).$
\end{defn}

 Suppose $\mathcal{M}$
is a field. We can \textbf{normalize} a Puiseux series $p= \sum
p_{\mathbf i}\la_1^{i_1} \cdots \la_n^{i_n}$ at any given $\mathbf i
\in \supp( p )$ by dividing through by $p_{\mathbf i}$; then the
normalized coefficient is $\one$. Given two Puiseux series $p= \sum
p_{\mathbf i}\la_1^{i_1} \cdots \la_n^{i_n}$, $q = \sum q_{\mathbf
i}\la_1 ^{i_1} \cdots \la_n^{i_n}$ having a common monomial
$\Lambda^{\mathbf i} = \la_1 ^{i_1} \cdots \la_n^{i_n}$ in their
support, one can normalize both and assume that $ p_{\mathbf i} =
q_{\mathbf i} = \one,$ and
 remove this
monomial from their difference $p-q$, i.e., the coefficient of
$\Lambda^{\mathbf i}$ in $\val(p-q)$ is $\zero ( = - \infty  ).$

Accordingly, a  \textbf{tropical ideal} of $\Gamma[\Lambda]$ is an
ideal $\mathcal I$ such that for any two polynomials $f= \sum
f_{\mathbf i}\la_1^{i_1} \cdots \la_n^{i_n}$, $g= \sum f_{\mathbf
i}\la_1^{i_1} \cdots \la_n^{i_n} \in \mathcal I$ having a common
monomial $\Lambda^{\mathbf j}$   there is $h =  \sum h_{\mathbf
i}\la_1^{i_1} \cdots \la_n^{i_n}\in \mathcal I$ whose coefficient of
$\Lambda^{\mathbf j}$ is $\zero,$ for which $$ h_{\mathbf i} \ge
\min \{a_f f _{\mathbf i}, b_g g_{\mathbf i}\}, \, \forall {\mathbf
i} $$ for suitable $a_f,b_g \in \mathcal{M}.$

For any tropical ideal $\mathcal I$, the sets of minimal indices of
supports constitutes the set of circuits of a matroid. This can be
formulated in terms of valuated matroids, defined in
\cite[Definition~1.1]{DW0} as follows:

  A
\textbf{valuated matroid} of \textbf{rank} $m$ is a pair $(E,v)$
where  $E$ is a set and $v: E^{(m)} \to \Gamma$   is a   map
satisfying the following properties:

\begin{enumerate}\eroman\item    There exist $e_1, \dots, e_m \in E$ with $v(e_1, \dots, e_m) \ne 0.$
 \item
 $v(e_1, \dots, e_m)= v(e_{\pi(1)}, \dots, e_{\pi(m)})$. for each $e_1, \dots, e_m \in E$ and every
permutation $\pi$. Furthermore,  $v(e_1, \dots, e_m)=\zero$ in case
some $e_i = e_j$.  \item For $(e_0, \dots, e_m,  e_2', \dots, e_m'
\in E$ there exists some $i$ with $1 \le i \le m$ and $$v(e_1,
\dots, e_m)v(e_0,  e_2', \dots, e_m')\le v(e_0, \dots,
e_{i-1},e_{i+1}, \dots, e_m)v(e_i,  e_2', \dots, e_m').$$
\end{enumerate}

This information is encapsulated in the following result.

 \cite[Theorem 1.1]{MacR} Let $K$ be a field with a valuation $\val : K \to \Gz$, and let
 $Y$
be a closed subvariety of~${K^\times}^{(n)}$ defined by an ideal
$\mathcal I \triangleleft K[\la _1^{\pm 1}  , \dots,\la _n^{\pm 1}
].$ Then any of the following three objects determines the others:
\begin{enumerate}\eroman\item The congruence $\Trop(\mathcal I)$ on the semiring
$\mathcal S := \Gamma[\la _1^{\pm 1}  , \dots,\la _n^{\pm 1}  ]$ of
tropical Laurent polynomials; \item The ideal $\tropa(\mathcal I)$
in $  S$; \item  The set of valuated matroids of the vector spaces
$\mathcal I_h ^ d ,$ where $\mathcal I_h ^ d $ is the degree $d$
part of  the homogenization of the tropical ideal $\mathcal I$.
 \end{enumerate}

\subsubsection{The supertropical approach}$ $

In supertropical mathematics the definitions run somewhat more
smoothly. $V$ was defined in terms of $\tT^\circ$ so for $f,g \in
\tT [\Lambda]$ we define the \textbf{$\circ$-equivalence} $f
\equiv_\circ g$ on $\tT [\Lambda]$ if and only if $f^\circ = g^\circ
$, i.e. $f(a)^\circ = g(a)^\circ$ for each $a \in \tT .$

 \begin{prop}\label{epicspl} The bend relation is the same as the $\circ$-equivalence,
 in the sense that $f \ \equiv_{\operatorname{bend}}\
 g$ iff $f \equiv_\circ g$, for any polynomials in $f,g \in \tT [\Lambda]$.
 \end{prop}  \begin{proof} The bend relation is obtained from a
 sequence of steps, each removing or adding on a monomial which
 takes on the same value of some polynomial $f$ on $V$. Thus the
 defining relations of the bend relation are all $\circ$ relations.
 Conversely, given a $\circ$ relation $f^\circ = g^\circ $, where
 $f = \sum f_i$ and $g = \sum g_j$ for monomials $f_i, g_j$,
 we have $$f\ \equiv_{\operatorname{bend}}\  g_1 + f \ \equiv_{\operatorname{bend}}\  g_1 + g_2 + f
 \ \equiv_{\operatorname{bend}}\  \dots \ \equiv_{\operatorname{bend}}\  g  + f \ \equiv_{\operatorname{bend}}\  g
 +  f_{\hat
 h} \ \equiv_{\operatorname{bend}}\  \dots \ \equiv_{\operatorname{bend}}\  g, $$
 implying $f \ \equiv_{\operatorname{bend}}\
 g$. (The same sort of argument is given in the proof of
 \cite[Proposition~2.6]{PR}.)

  \end{proof}

In supertropical algebra, given a   polynomial $f = \sum \a_{\mathbf
i}\Lambda ^\mathbf i ,$  define  $\vsupp( f )$ to be all the tuples
$\mathbf i =({i_1}\cdots  {i_n})$  for which the monomial in $f$ has
 coefficient $\a_{\mathbf i}\in \tT$.

 The supertropical version of tropical
ideal is that if $f,g \in \mathcal I$ and $\mathbf i \in \vsupp( f )
\cap \vsupp( g ),$ then, by normalizing, there are $a_f, b_g$ such
that  $\mathbf i \notin \vsupp(a_f f + b_g g)$. This is somewhat
stronger than the claim of the previous paragraph, since it
specifies the desired element.

Supertropical ``$d$-bases'' over a super-semifield are treated in
\cite{IKR}, where vectors are defined to be independent iff no
tangible linear combination is a ghost. If $V$ is defined as the set
 $v \in F^{(n)}: f_j(v) \in \nu(F),\ \forall j \in J$ for a set $\{ f_j : j \in J\}$ of homogeneous polynomials of  degree
 $m$, then taking $\mathcal I = \{ f: f(V) \in \nu(F)\}$ and  $\mathcal
 I_m$ to be its polynomials of degree $m$, one sees that the
$d$-bases of $I_m$ of cardinality $m$  comprise a matroid (whose
circuits are those polynomials of minimal support), by
\cite[Lemma~4.10]{IKR}. On the other hand, submodules of free
modules can fail to satisfy Steinitz' exchange property
(\cite[Examples~4.18,4.9]{IKR}), so there is room for considerable
further investigation.

``Supertropicalization'' then is the same tropicalization map as
$\tropa$, now taken to the standard supertropical semifield $\strop:
\Gamma \cup \tGz$ (where $\tT = \Gamma$). In view of
Proposition~\ref{epicspl}, the analogous proof of \cite[Theorem
1.1]{MacR} yields the corresponding result:

\begin{thm} Let $K$ be a field with a valuation $\val : K \to \Gz$, and
let
 $Y$
be a closed subvariety of  ${K^\times}^{(n)}$ defined by an ideal $
I \triangleleft K[\la _1^{\pm 1}  , \dots,\la _n^{\pm 1}  ].$ Then
any of the following  objects determines the others:
\begin{enumerate}\eroman\item The congruence given by $\circ$-equivalence on $\mathcal
I = \strop(I)$ in the supertropical \semiring0 $\mathcal S :=
\Gamma[\la _1^{\pm 1} , \dots,\la _n^{\pm 1}  ]$ of tropical Laurent
polynomials; \item The ideal $\tropa(\mathcal I)$ in $\mathcal  S$;
\item  The set of valuated matroids of the vector spaces $\mathcal
I_h ^ d ,$ where $\mathcal I_h ^ d $ is the degree $d$ part of  the
homogenization of the tropical ideal $\mathcal I$.
 \end{enumerate}
 \end{thm}

\subsubsection{The systemic approach}$ $

The supertropical approach can be generalized directly to the
systemic approach, which also includes hyperfields and fuzzy rings.
We assume $(\mathcal A, \tT, (-), \preceq)$ is a system.

\begin{defn}\label{sys1}
The \textbf{$\circ$-equivalence} on $\Fun (S, \mathcal A)$ is
defined by,  $f \equiv_\circ g$ if and only if $f^\circ = g^\circ $,
i.e. $f(s)(b)^\circ = g(s)(b)^\circ$ for each $s \in S$ and $b \in
\mathcal A.$
\end{defn}

This matches the supertropical definition.

\begin{defn}\label{sys2}
Given  $f \in \Fun (S, \mathcal A)$ define $\circ$-$\supp( f )= \{ s
\in S: f(s) \in \tT\}$.
\end{defn}

 The systemic version of tropical
ideal is that if $f,g \in \mathcal I$ and $s \in \circ$-$\supp( f
)\, \cap\, \circ$-$\supp( g),$ then there are $a_f, b_g \in \tT$
such that
  $s \notin \circ$-$\supp (a_f  f (-) b_g   g)$.

Now one can view tropicalization as a functor as in \cite[\S
8]{Row16}.

\section{Areas for further research}

\subsection{Geometry}

A \textbf{$\preceq$-root} of a polynomial $f \in \mathcal A [\la_1,
\dots, \la_n]$ is some $n-$-tuple $(a_1, \dots, a_n)$ such that
$f(a_1, \dots, a_n) \succeq \zero.$ This leads naturally to affine
$\preceq$-varieties (as common $\preceq$-roots of a set of
polynomials), and algebraic geometry. An alternative approach is
through Hopf semi-algebras.

\subsection{Valuated  matroids over systems}
Viewing tropicalization as a functor, define the appropriate
valuated matroid. Then one can address the recent work on matroids
and valuated matroids, and formulate them over systems in analogy to
\cite{DW}. Presumably, as in \cite{AGR},   in the presence of
various assumptions, one could carry out the proofs of many of these
assertions.

\end{document}